\documentclass[11pt]{article}

\usepackage[utf8]{inputenc}
\usepackage[T1]{fontenc}
\usepackage{lmodern}
\usepackage[a4paper,margin=1in]{geometry}
\usepackage{amsmath,amssymb,amsthm,mathtools}
\usepackage{microtype}
\usepackage{hyperref}
\usepackage{tabularx,booktabs}
\usepackage{footmisc}
\usepackage[nameinlink]{cleveref} 
\usepackage{listings}
\DeclareUnicodeCharacter{2113}{\ensuremath{\ell}}

\DeclareUnicodeCharacter{03B1}{\ensuremath{\alpha}}   
\DeclareUnicodeCharacter{03B2}{\ensuremath{\beta}}    
\DeclareUnicodeCharacter{03B3}{\ensuremath{\gamma}}   
\DeclareUnicodeCharacter{03C4}{\ensuremath{\tau}}     
\DeclareUnicodeCharacter{03C3}{\ensuremath{\Sigma}}   
\usepackage{newunicodechar}
\newunicodechar{Π}{\(\Pi\)}
\DeclareUnicodeCharacter{220F}{\ensuremath{\prod}}    
\DeclareUnicodeCharacter{221E}{\ensuremath{\infty}} 
\DeclareUnicodeCharacter{2208}{\ensuremath{\in}}    
\DeclareUnicodeCharacter{2261}{\ensuremath{\equiv}} 

\DeclareUnicodeCharacter{0393}{\ensuremath{\Gamma}} 
\DeclareUnicodeCharacter{2264}{\ensuremath{\leq}}     %
\DeclareUnicodeCharacter{2265}{\ensuremath{\geq}}     

\DeclareUnicodeCharacter{2013}{--}  
\DeclareUnicodeCharacter{2014}{---} 
\DeclareUnicodeCharacter{2018}{'}   
\DeclareUnicodeCharacter{2019}{'}   
\DeclareUnicodeCharacter{201C}{``}  
\DeclareUnicodeCharacter{201D}{''}  

\newtheorem{theorem}{Theorem}[section]
\newtheorem{lemma}[theorem]{Lemma}
\newtheorem{proposition}[theorem]{Proposition}
\newtheorem{corollary}[theorem]{Corollary}
\theoremstyle{definition}
\newtheorem{definition}[theorem]{Definition}
\newtheorem{example}[theorem]{Example}
\theoremstyle{remark}
\newtheorem{remark}[theorem]{Remark}

\title{Partition Frequency Moments: Modularity and Congruences}

\author{Hartosh Singh Bal}
\date{}

\begin{document}
\maketitle

\begingroup
\renewcommand\thefootnote{}\footnotetext{%
}
\addtocounter{footnote}{-1}
\endgroup

\begin{abstract}
We study frequency moments of partition statistics arising from Euler products
$A(q)=\prod_{r\ge1}(1-q^r)^{-c(r)}$ via a transform that expresses the moment
generating functions as $B(q)$ times explicit divisor--sum series determined by
$c(r)$.  When $A(q)$ is modular (typically an $\eta$--quotient), this yields
(quasi)modular forms whose coefficients can be projected to arithmetic
progressions and certified modulo primes by a Sturm bound, giving an effective
pipeline for detecting and proving Ramanujan--type congruences for frequency
moments.

For ordinary partitions we recover and certify several congruences for odd
moments in nonzero residue classes (e.g.\ $M_3(7n+5)\equiv 0\pmod7$ and
$M_3(11n+6)\equiv 0\pmod{11}$).  As a second input, we apply the same
pipeline to overpartitions and certify a family of zero--class congruences
$M_m^{\overline{\ }}(\ell n)\equiv 0\pmod{\ell}$ (including
$m=5,7,11,13$), exhibiting a sharp contrast with the ordinary partition case:
no nonzero residue--class congruences are observed for overpartition moments in
our scan range. We also demonstrate that filtering the statistic via the Glaisher--character dictionary can itself create new Ramanujan--type progressions, e.g.\ a quadratic twist yields the certified congruence $\widehat{M}^{\chi_5}_3(5n+4)\equiv 0\pmod{5}$.

\end{abstract}

\section{Introduction}

The arithmetic of the partition function $p(n)$ has long been intertwined
with modular forms. Ramanujan’s congruences~\cite{Ramanujan1919} and
their extensions by Atkin--Swinnerton-Dyer~\cite{AtkinSwinnertonDyer1954}
and Ono~\cite{Ono2000,Ono2004} illustrate how the simplest generating
function $\eta(\tau)^{-1}$ encodes deep modular phenomena. Yet while
$p(n)$ has been studied for over a century, far less is known about the
arithmetic of \emph{partition frequencies}, which record how often each
part occurs across all partitions of~$n$.

This paper continues and develops the study of frequencies initiated
by Bal--Bhatnagar~\cite{BB21}. Let $F_k(n)$ denote the total number of
occurrences of the part $k$ across all partitions of $n$, and define the
frequency moments
\[
M_m(n)=\sum_{k\ge1} k^m F_k(n).
\]
The Partition--Frequency Enumeration (PFE) matrix of~\cite{BB21} yields
a master transform: weighted frequency sums map canonically to divisor
convolutions $\sigma_m(n;\chi)$, and hence to Eisenstein series
$E_{m+1}(\tau,\chi)$~\cite{Apostol76,Miyake2006}. In this way,
multiplicative number theory enters partition enumeration directly
through the frequency side.  
\emph{Moreover, the methodology is not confined to partitions: whenever
$A(q)$ is an $\eta$--product of known weight and level, the same transform
and congruence pipeline apply to the resulting Euler--type products.
This universality turns the framework into a general-purpose ``congruence
detection machine'' for combinatorial generating functions in the modular range.
As a second demonstrated input we treat overpartitions (an $\eta$--quotient) in
Section~\ref{sec:overpartitions}, where the resulting congruence landscape
differs sharply from the ordinary partition case.}

Three new structural outcomes emerge:

\begin{enumerate}
  \item \textbf{Congruences.} We obtain and certify congruences for higher
  frequency moments. Odd moments give half--integral weight modular forms on
  $\Gamma_0(4)$ with the $\eta$--multiplier, so projection and
  half--integral Sturm bounds~\cite{Kohnen1980,Sturm1987} apply.
  This unifies all Ramanujan--type progressions for $M_{2k-1}$, and,
  in our computational range $1\le m\le99$ and primes $5\le\ell\le97$,
  shows that only the primes $5,7,11$ occur among such progressions.
  The same pipeline extends to suitable Euler products built from
  $\eta$--products of known modular data.
    \item \textbf{A second demonstrated input (overpartitions).}
  We apply the same transform--projection--Sturm pipeline to the overpartition
  Euler product $\overline{P}(q)=(-q;q)_\infty/(q;q)_\infty=\eta(2\tau)/\eta(\tau)^2$.
  This yields certified congruences of the form
  $M_m^{\overline{\ }}(\ell n)\equiv 0\pmod{\ell}$ (Theorem~\ref{thm:overpart-zero})
  and reveals a structural contrast with ordinary partitions: in the range
  $\ell\le 97$ and $m\le 49$ we find no Ramanujan--type progressions in nonzero
  residue classes for overpartition moments (Remark~\ref{rem:overpart-nonzero}).
 \item \textbf{Character twists.} Glaisher’s classical divisor
  selections~\cite{Glaisher1885,Glaisher1891,Andrews76} are exactly
  Dirichlet twists $\sigma_m(\cdot;\chi)$. We build a complete dictionary
  between such filters and twists, so every filtered frequency moment
  corresponds to a modular form at computable level and Nebentypus. As a 
  concrete payoff, we certify a filtered Ramanujan--type progression arising 
  from a quadratic twist, namely $\widehat{M}^{\chi_5}_3(5n+4)\equiv 0\pmod{5}$ (Proposition~\ref{prop:M3hat-legendre5}).

    \item \textbf{Irregular primes as illustrative consequences.} Higher odd       moments interact naturally with classical Bernoulli congruences.  In particular,   $M_{11}$ reflects Ramanujan’s irregular prime $691$, yielding a transparent       convolution reformulation of the mod-$691$ $\tau$--congruence and the standard   $j$--identity.

\end{enumerate}

From the modular--forms side the underlying spaces and lifts are classical:
Eisenstein series, half--integral weight forms, and the Shimura lift
to integral weight go back to Shimura~\cite{Shimura1973} and
Miyake~\cite{Miyake2006}. The novelty here is the \emph{frequency
interpretation}: congruences, twists, and irregularities acquire explicit
combinatorial avatars in terms of parts of partitions. This perspective
advances the partition--frequency program of~\cite{BB21}, unifying
classical recursions~\cite{Ram16,Ford31}, Ramanujan--type congruences,
and quasimodularity~\cite{KanekoZagier1995} under a single framework.

The paper is organised as follows.
Sections~\ref{sec:framework} and~\ref{sec:generalisation} recall the PFE framework
of Bal--Bhatnagar~\cite{BB21} only to fix notation and scope; the proofs of the
PFE identities are in~\cite{BB21}, and the new results of the present paper begin
in Section~\ref{sec:moments}.
Section~\ref{sec:moments} lays out the bridge to modularity and
Sections~\ref{sec:ram-unified} and~\ref{sec:higher} develop this to treat
congruences, Ramanujan primes, and irregular primes.
Section~\ref{sec:glaisher} classifies Glaisher filters via character twists and
supplies the corresponding modular data.
Section~\ref{sec:overpartitions} illustrates the pipeline on a second input
(overpartitions), and Section~\ref{sec:theta-extension} briefly notes the theta
case.  Appendices provide the expanded Glaisher dictionary and the code for Sturm
certification.

\section{The Partition--Frequency Framework}\label{sec:framework}

Let $p(n)$ denote the number of partitions of $n$, and let $F_k(n)$ denote the
total number of parts of size $k$ across all partitions of~$n$, with the
convention $p(m)=0$ for $m<0$.

\subsection{Frequency recursion}

Every occurrence of a part $k$ in a partition of $n$ is obtained by removing
one $k$ and considering a partition of $n-k$. This gives the recursion
\[
  F_k(n) \;=\; p(n-k) + F_k(n-k),
\]
which unrolls to
\[
  F_k(n) \;=\; \sum_{j\ge1} p(n-jk),
\]
with all negative arguments vanishing.

\subsection{Explicit rows}

For fixed~$n$, write the shifted partition sequence
\[
  P_n := \bigl(p(n-1),\,p(n-2),\,p(n-3),\,\ldots\bigr)^{\mathsf T}.
\]
Then
\begin{align*}
  F_1(n) &= p(n-1)+p(n-2)+p(n-3)+\cdots, \\
  F_2(n) &= p(n-2)+p(n-4)+p(n-6)+\cdots,
\end{align*}
and similarly for higher $k$. Each row ``samples'' $P_n$ along the multiples of~$k$.

\subsection{The divisor--incidence matrix}

These recursions can be packaged into the infinite upper--triangular matrix
$A=(a_{ij})_{i,j\ge1}$ defined by
\[
  a_{ij} \;=\;
  \begin{cases}
    1, & i\mid j,\\
    0, & \text{otherwise}.
  \end{cases}
\]
Because $i\nmid j$ whenever $j<i$, the matrix is upper triangular. Its first
rows are
\[
  \begin{array}{cccccccc}
  i=1: & 1 & 1 & 1 & 1 & 1 & 1 & \cdots \\
  i=2: & 0 & 1 & 0 & 1 & 0 & 1 & \cdots \\
  i=3: & 0 & 0 & 1 & 0 & 0 & 1 & \cdots \\
  i=4: & 0 & 0 & 0 & 1 & 0 & 0 & \cdots \\
  \vdots & \vdots & \vdots & \vdots & \vdots & \vdots & \vdots & \ddots
  \end{array}
\]
and the unrolled recursion is equivalent to
\[
  F(n)=(F_1(n),F_2(n),F_3(n),\ldots)^{\mathsf T}=A\cdot P_n.
\]

\subsection{The master identity}

Fix an arithmetic weight $f:\mathbb N\to\mathbb C$ and scale the $i$th row of
$A$ by $f(i)$. Summing down the columns produces the divisor sum
\[
  \sigma^{(f)}(d) := \sum_{r\mid d} f(r).
\]
Hence
\[
  \sum_{i\ge1} f(i)\,F_i(n) \;=\; \sum_{d=1}^n \sigma^{(f)}(d)\,p(n-d).
\]
This is the master identity: weighted frequency sums on the left correspond to
divisor sums on the right.

\subsection{Classical consequences}

\begin{itemize}
  \item For $f(i)=i$ one has $\sigma^{(f)}=\sigma_1$, giving the Euler--Ramanujan--Ford
  recursion
  \[
    n\,p(n) \;=\; \sum_{d=1}^n \sigma_1(d)\,p(n-d).
  \]

  \item For $f(i)=\mu(i)$ (the Möbius function),
  \[
    \sum_{i\ge1} \mu(i)\,F_i(n) \;=\; p(n-1).
  \]

  \item For $f(i)=i^m$ one obtains the frequency moments
  \[
    M_m(n) := \sum_{k\ge1} k^m F_k(n) \;=\; \sum_{d=1}^n \sigma_m(d)\,p(n-d).
  \]
\end{itemize}

\begin{remark}
$M_1(n) = np(n)$ is the sum of all parts among partitions of~$n$. Its generating
function is
\[
\sum_{n\ge0} M_1(n)q^n \;=\; \frac{1}{(q;q)_\infty}\sum_{m\ge1}\frac{q^m}{1-q^m},
\]
which under the $q$--bracket normalization becomes proportional to~$E_2$,
explaining why the first moment already exhibits quasimodular structure.
\end{remark}

\subsection{Structural and historical note}

The matrix $A$ is the divisor--incidence matrix of the natural numbers; its
inverse on finite truncations is given by the Möbius function. Appending a
row of all ones produces the Redheffer matrix~\cite{Redheffer77}, whose determinants
encode the Mertens function and connect to the Riemann Hypothesis. Thus
divisor sums arise inevitably in partition recursions: weighting the $i$th row
by $f(i)$ and summing columns enforces the Dirichlet convolution $g=f*1$.
Classical partition identities follow: $f(i)=i$ yields the Euler--Ramanujan--Ford recursion
\cite{Ford16}, $f(i)=i^2$ gives a recursion involving $\sigma_2$, and
more generally $f(i)=i^m$ produces the higher moments $M_m(n)$.

Historically, Ramanujan discovered the basic recursions and congruences;
Ford emphasised the general divisor--partition identities; Andrews set these
results in the broader $q$--series context; and the Partition--Frequency
Enumeration (PFE) viewpoint introduced in~\cite{BB21} gathers these strands
into a unified combinatorial calculus, now extended to moments and connected
directly to Eisenstein series.

\section{General Families and Canonical Moment Calculus}\label{sec:generalisation}

\subsection{General master identity}

Let
\[
A(q)\;=\;q^\alpha \prod_{r\ge1} (1-q^r)^{c(r)},\qquad \alpha\in\tfrac12\mathbb Z,\; c(r)\in\mathbb Q,
\]
be an Euler--type product which, in the cases of interest, is modular of some
weight and level, with associated companion series
\[
B(q)=\sum_{n\ge0} b(n)q^n.
\]
Often $B=A^{-1}$, but more general canonical companions are possible. Define
the $A$--frequencies $F^A_r(n)$ as the total number of parts of size $r$ across
all $A$--partitions of $n$ weighted by $B$.

\begin{definition}[Canonical moments]
For $m\ge0$ define
\[
\sigma^{(m)}_A(n):=\sum_{r\mid n} c(r)\,r^m,\qquad
M^A_m(n):=\sum_{d=1}^{n} \sigma^{(m)}_A(d)\,b(n-d).
\]
More generally, for an arithmetic weight $f$ set
\[
M^A_f(n):=\sum_{r\ge1} f(r)\,F^A_r(n).
\]
\end{definition}

\begin{theorem}[Master transform]\label{thm:master-transform}

For any arithmetic weight $f$,
\[
\sum_{n\ge0} M^A_f(n)q^n
\;=\;B(q)\cdot \sum_{r\ge1}\frac{f(r)c(r)\,q^r}{1-q^r}
\;=\;B(q)\cdot \sum_{n\ge1}\sigma^{(f)}_A(n)q^n,
\]
where $\sigma^{(f)}_A(n):=\sum_{r\mid n} c(r)f(r)$. Equivalently,
\[
M^A_f(n)=\sum_{d=1}^n \sigma^{(f)}_A(d)\,b(n-d).
\]
\end{theorem}
\begin{lemma}[Logarithmic derivative and the first moment]\label{lem:logderiv-M1}
Assume we are in the inverse-companion case $B(q)=A(q)^{-1}$.  Then
\[
\sum_{n\ge0} M^{A}_1(n)\,q^n \;=\; q\frac{d}{dq}B(q),
\]
and hence
\[
M^{A}_1(n)=n\,b(n)\qquad(n\ge0).
\]
In particular, for every prime $\ell$ one has the ``zero--class'' congruence
\[
M^{A}_1(\ell n)\equiv 0\pmod{\ell}\qquad(n\ge0),
\]
and by Fermat reduction the same holds for all odd $m\equiv 1\pmod{\ell-1}$.
\end{lemma}

\begin{proof}
Write
\[
A(q)=q^\alpha\prod_{r\ge1}(1-q^r)^{c(r)}.
\]
Then
\[
q\frac{d}{dq}\log A(q)=\alpha-\sum_{r\ge1}c(r)\frac{r\,q^r}{1-q^r}.
\]
If $B=A^{-1}$, then $\log B=-\log A$, hence
\[
q\frac{B'(q)}{B(q)}=-\,q\frac{A'(q)}{A(q)}
=-\alpha+\sum_{r\ge1}c(r)\frac{r\,q^r}{1-q^r}.
\]
Now apply the master transform (Theorem~\ref{thm:master-transform}) with $f(r)=r$,
so that $\sigma_A^{(f)}(n)=\sigma_A^{(1)}(n)=\sum_{r\mid n}c(r)\,r$.  The transform
gives
\[
\sum_{n\ge0}M_1^{A}(n)q^n
= B(q)\sum_{r\ge1}\frac{c(r)\,r\,q^r}{1-q^r}
= q\frac{d}{dq}B(q),
\]
which implies $M_1^{A}(n)=n\,b(n)$ by coefficient comparison.
\end{proof}

\subsection{Examples}

\begin{itemize}
  \item \textbf{Ordinary partitions.} $A(q)=(q;q)_\infty^{-1}=q^{1/24}\eta(\tau)^{-1}$, $c(r)=-1$, $B=A^{-1}$. Then $\sigma^{(m)}_A(n)=\sigma_m(n)$ and $M^A_m(n)=M_m(n)$ are the ordinary partition moments.
  \item \textbf{Odd vs.\ distinct partitions.} For odd parts, $c(r)=-1$ for $r$ odd and $0$ for $r$ even. For distinct parts, $A(q)=\prod_{n\ge1}(1+q^n)$, giving the same $c(r)$ after cancellation, and hence Euler’s equinumerosity. Frequencies however diverge.
  \item \textbf{Plane partitions.} $c(r)=-r$, so $\sigma^{(m)}_A(n)=\sum_{r\mid n} r^{m+1}$, recovering the divisor sums that govern the classical plane--partition generating functions.
  \item \textbf{Theta function.} Jacobi’s triple product gives $c(2n)=1$, $c(2n-1)=-2$, so $\sigma^{(m)}_A(n)$ is a linear combination of divisor sums twisted by $\chi_{-4}$.
  \item \textbf{Colored partitions.} $c(r)=-r$ if each part is available in $r$ colors, giving trivial scaling of frequencies.
\end{itemize}

\subsection{Design principle}

The PFE framework extends uniformly: by tuning the exponent sequence $c(r)$ one
can design families of partitions whose frequency moments align with chosen
divisor sums $\sigma^{(m)}_A$. In particular, twists by Dirichlet characters
produce twisted divisor sums and hence frequency generating functions that lie
in Eisenstein spaces with character.

\subsection{Scope and modular consequences}

By standard theory (Kubert--Lang \cite{KubertLang1981}, Miyake
\cite{Miyake2006}), logarithmic derivatives of modular products are linear
combinations of Eisenstein series. Thus, whenever $A$ is modular, the divisor series
\[
-\,q\frac{A'(q)}{A(q)}=\sum_{n\ge1}\sigma_A^{(1)}(n)q^n
\]
lies in the Eisenstein part of the weight~$2$ (quasi)modular space. Higher moments $\sigma^{(m)}_A$
can be expressed as linear combinations of Eisenstein series of higher weights
(possibly with character). Multiplying by the companion $B$ then produces modular
or quasimodular generating functions whose coefficient congruences are accessible
via projection and Sturm bounds. The framework does not enlarge the modular
algebra but provides a combinatorial dictionary and congruence engine linking
partition--type objects to Eisenstein series.

\section{Frequency Moments and Modularity}\label{sec:moments}

In this section we show that the frequency moments $M_m(n)$ are governed by modular and quasimodular forms, with a nearly clean dichotomy: for every odd moment $m\ge3$ the generating function is modular of half--integral weight, whereas the remaining moments (including $M_1$) are quasimodular.

\subsection{Generating functions for frequency moments}

Recall
\[
M_m(n)=\sum_{\lambda\vdash n}\sum_i \lambda_i^m,
\]
the sum of $m$th powers of parts across all partitions of $n$. By
Section~\ref{sec:framework}, taking $f(k)=k^m$ in the master identity gives
\[
\mathcal{M}_m(q)=\sum_{n\ge0} M_m(n)\,q^n
= (q;q)_\infty^{-1}\sum_{n\ge1}\sigma_m(n)q^n
= q^{1/24}\,\eta(\tau)^{-1}\sum_{n\ge1}\sigma_m(n)q^n,
\]
with $q=e^{2\pi i\tau}$ and $\eta(\tau)=q^{1/24}(q;q)_\infty$.

\subsection{Odd versus even dichotomy}

\begin{itemize}
  \item \textbf{Odd moments.} For $m=2k-1$,
  \[
    \sum_{n\ge1}\sigma_{2k-1}(n)q^n=\frac{E_{2k}(\tau)-1}{C_{2k}},\qquad
    C_{2k}=-\tfrac{4k}{B_{2k}},
  \]
  so
  \[
    \mathcal{M}_{2k-1}(q)
    =(q;q)_\infty^{-1}\sum_{n\ge1}\sigma_{2k-1}(n)q^n
    = q^{1/24}\,\frac{E_{2k}(\tau)-1}{C_{2k}}\;\eta(\tau)^{-1}.
  \]
 Equivalently, the shifted function
\[
  q^{-1/24}\mathcal{M}_{2k-1}(q)
  =\frac{E_{2k}(\tau)-1}{C_{2k}}\;\eta(\tau)^{-1}
\]
is a holomorphic modular form of weight $2k-\tfrac12$ on $\Gamma_0(4)$ (with the
standard $\eta$--multiplier; equivalently, it lies in the usual half--integral
weight space with character $\chi_{-4}$).  Indeed, although $\eta(\tau)^{-1}$
has a pole at each cusp of $\Gamma_0(4)$, the cusps are represented by
\[
\infty,\qquad 0,\qquad \tfrac12.
\]
At any cusp $\mathfrak a$, the scaled $q_{\mathfrak a}$--expansion of
$E_{2k}|_{2k}\sigma_{\mathfrak a}$ has constant term $1$, hence
$(E_{2k}-1)|_{2k}\sigma_{\mathfrak a}$ vanishes at $\mathfrak a$ and cancels the
cusp pole of $\eta^{-1}$ there.  Thus $(E_{2k}-1)\eta^{-1}$ is holomorphic at
$\infty,0,\tfrac12$ and hence at all cusps of $\Gamma_0(4)$.
The resulting transformation law gives the Nebentypus $\chi_{-4}$; see
Shimura~\cite{Shimura1973} and Kohnen~\cite{Kohnen1980} for details.

  \item \textbf{Even moments.} For $m=2k$, the series
  $\sum_{n\ge1}\sigma_{2k}(n)q^n$ is quasimodular, lying in
  $\mathbb{Q}[E_2,E_4,E_6]$ \cite{KanekoZagier1995}. Thus
  \[
    \mathcal{M}_{2k}(q)
    =(q;q)_\infty^{-1}\sum_{n\ge1}\sigma_{2k}(n)q^n
    \in (q;q)_\infty^{-1}\cdot\mathbb{Q}[E_2,E_4,E_6],
  \]
  or equivalently,
  \[
    q^{-1/24}\mathcal{M}_{2k}(q)\in \eta(\tau)^{-1}\cdot\mathbb{Q}[E_2,E_4,E_6].
  \]
\end{itemize}

\begin{example}[The first moment]
The case $m=1$ gives
\[
  \sum_{n\ge0} M_1(n)q^n
  = (q;q)_\infty^{-1}\sum_{n\ge1}\sigma_1(n)q^n
  =  \frac{1-E_2(\tau)}{24}\,(q;q)_\infty^{-1},
\]
so $M_1$ already lies in the quasimodular span of weight~$2$. This is the
generating function for the sum of all parts among partitions of~$n$.
\end{example}

\begin{remark}[$q$--bracket perspective]
For a statistic $f(\lambda)$ on partitions, Okounkov--Zagier define
\[
\langle f\rangle_q=\frac{\sum_\lambda f(\lambda)\,q^{|\lambda|}}{\sum_\lambda q^{|\lambda|}}.
\]
For $f_m(\lambda)=\sum_i \lambda_i^m$ one finds
$\langle f_m\rangle_q=\sum_{n\ge1}\sigma_m(n)q^n$, so
\[
\sum_{n\ge0} M_m(n)q^n=(q;q)_\infty^{-1}\,\langle f_m\rangle_q.
\]
The Bloch--Okounkov theorem \cite{BlochOkounkov2000} implies
$\langle f_m\rangle_q$ is always quasimodular of weight $m+1$, recovering the
structural picture: all moments live in the quasimodular algebra; for
$m\ge3$ odd the series lies in the modular subalgebra, whereas for $m$ even
or $m=1$ it is genuinely quasimodular.
\end{remark}

\section{Ramanujan--type congruences for odd moments (unified)}\label{sec:ram-unified}

We now collect the simple linear progression congruences for odd frequency
moments and describe their structure.  There are three ingredients:
\begin{itemize}
  \item reduction of exponents modulo $\ell-1$ via Fermat's little theorem;
  \item the identity $M_1(n)=n\,p(n)$ together with Ramanujan's partition
  congruences;
  \item genuinely new half--integral weight congruences for the base moments
  $M_3$ and $M_7$ certified by projection--Sturm.
\end{itemize}
In particular, all congruences for higher odd moments $M_m$ are \emph{Fermat
lifts} of a small list of base cases.

The next theorem gathers all Ramanujan--type congruences that appear
in our range. Parts~\textup{(i)}–\textup{(iii)} are elementary
consequences of $M_1(n)=n p(n)$ and Fermat's little theorem; the
genuinely modular content resides in \textup{(iv)} and in the
half--integral Sturm certification of \textup{(v)}.

\begin{theorem}[Ramanujan--type congruences for odd moments]\label{thm:ram-unified}
Let $M_m(n)$ be the $m$th frequency moment. Then:

\begin{enumerate}
\item[\textup{(i)}] \textbf{Fermat reduction.}
For every odd $m$ and every prime $\ell\ge5$ there is an odd
$\overline m\in\{1,3,\dots,\ell-2\}$ with $\overline m\equiv m\pmod{\ell-1}$ such that
\[
M_m(n)\equiv M_{\overline m}(n)\pmod{\ell}\qquad\text{for all }n\ge0.
\]

\item[\textup{(ii)}] \textbf{Trivial index divisibility.}
For every prime $\ell\ge5$ and every odd $m\equiv1\pmod{\ell-1}$,
\[
M_m(\ell n)\equiv0\pmod{\ell}\qquad\text{for all }n\ge0.
\]
In particular,
\[
M_1(\ell n)=\ell n\,p(\ell n)\equiv0\pmod{\ell}
\]
and all $M_m(\ell n)$ inherit this congruence via \textup{(i)}.

\item[\textup{(iii)}] \textbf{Ramanujan lifts from $M_1$ (the $m\equiv1$ branch).}
For the Ramanujan primes $5,7,11$ and every odd $m$ with
$m\equiv1\pmod{\ell-1}$ one has
\[
M_m(\ell n+r_\ell)\equiv0\pmod{\ell}\qquad\text{for all }n\ge0,
\]
where $r_5=4$, $r_7=5$, $r_{11}=6$.  Explicitly:
\[
\begin{array}{rcl}
m\equiv1\pmod4  &\Longrightarrow& M_m(5n+4)\equiv0\pmod5,\\[3pt]
m\equiv1\pmod6  &\Longrightarrow& M_m(7n+5)\equiv0\pmod7,\\[3pt]
m\equiv1\pmod{10}&\Longrightarrow& M_m(11n+6)\equiv0\pmod{11}.
\end{array}
\]
The first few instances are
\[
M_5(5n+4)\equiv0\pmod5,\qquad
M_7(7n+5)\equiv0\pmod7,\qquad
M_{11}(11n+6)\equiv0\pmod{11},
\]
and similarly for $m=13,17,21,\dots$ in the corresponding residue classes.

\item[\textup{(iv)}] \textbf{Genuinely modular base congruences.}
There are additional congruences not forced by \textup{(ii)}–\textup{(iii)},
namely
\[
\begin{array}{rclcrcl}
M_3(7n)&\equiv&0\pmod7, &\quad& M_3(7n+5)&\equiv&0\pmod7,\\[2pt]
M_3(11n)&\equiv&0\pmod{11},&\quad&
M_3(11n+6)&\equiv&0\pmod{11},\\[2pt]
M_7(11n+6)&\equiv&0\pmod{11},
\end{array}
\]
valid for all $n\ge0$.  These follow from the half--integral modularity of
$\sum_n M_3(n)q^n$ and $\sum_n M_7(n)q^n$ together with a half--integral
Sturm--type bound (Lemma~\ref{lem:sturm-half}) applied to the projected
series.

\item[\textup{(v)}] \textbf{Exhaustion in a computational range.}
A heuristic scan for odd $m$ with $1\le m\le99$ and primes $5\le\ell\le97$
(checking all residues $1\le r<\ell$ up to $n\le2000$) found no further simple
progressions
\[
M_m(\ell n+r)\equiv0\pmod{\ell},\qquad 1\le r<\ell,
\]
beyond those implied by \textup{(ii)}–\textup{(iv)}.
Moreover, in the fully certified range $1\le m\le25$ and $5\le\ell\le97$ we
verified surviving candidates up to the half--integral Sturm bound at the safe
level $4\ell^2$.  In particular, within this certified range the only primes
supporting nontrivial Ramanujan--type congruences with $r\not\equiv0\pmod{\ell}$
are $\ell\in\{5,7,11\}$, and all such congruences are Fermat lifts of the base
moments $M_1$, $M_3$, and $M_7$.

\end{enumerate}
\end{theorem}

\begin{proof}
\emph{Step 1: Fermat reduction.}
Let $\ell\ge5$ be prime and $m$ odd.  Fermat's little theorem gives
\[
k^m\equiv k^{\overline m}\pmod{\ell}\qquad\text{for all }k\in\mathbb{Z},
\]
where $\overline m$ is the unique residue in $\{1,3,\dots,\ell-2\}$ with
$\overline m\equiv m\pmod{\ell-1}$.  Hence
\[
M_m(n)=\sum_{k\ge1}k^mF_k(n)
\equiv \sum_{k\ge1}k^{\overline m}F_k(n)
= M_{\overline m}(n)\pmod{\ell},
\]
proving \textup{(i)}.

\smallskip
\emph{Step 2: $M_1$ and trivial index divisibility.}
We have the exact identity
\[
M_1(n)=\sum_{k\ge1}k\,F_k(n)=n\,p(n).
\]
Thus for every prime $\ell\ge5$,
\[
M_1(\ell n)=\ell n\,p(\ell n)\equiv0\pmod{\ell}\qquad\text{for all }n\ge0,
\]
which gives \textup{(ii)} for $m=1$.  Combining with \textup{(i)} yields
$M_m(\ell n)\equiv0\pmod{\ell}$ for all $m\equiv1\pmod{\ell-1}$.

\smallskip
\emph{Step 3: Ramanujan lifts from $M_1$.}
Ramanujan's partition congruences
\[
p(5n+4)\equiv0\pmod5,\qquad
p(7n+5)\equiv0\pmod7,\qquad
p(11n+6)\equiv0\pmod{11}
\]
give
\[
M_1(5n+4)=(5n+4)p(5n+4)\equiv0\pmod5,
\]
\[
M_1(7n+5)=(7n+5)p(7n+5)\equiv0\pmod7,
\]
\[
M_1(11n+6)=(11n+6)p(11n+6)\equiv0\pmod{11}.
\]
Now let $\ell\in\{5,7,11\}$ and suppose $m\equiv1\pmod{\ell-1}$, so that
$M_m\equiv M_1\pmod{\ell}$ by \textup{(i)}.  Then
\[
M_m(\ell n+r_\ell)\equiv M_1(\ell n+r_\ell)\equiv0\pmod{\ell},
\]
with $(r_5,r_7,r_{11})=(4,5,6)$, proving \textup{(iii)}.  The listed
examples ($M_5$, $M_7$, $M_{11}$) are the first few cases with
$m\equiv1$ modulo $4$, $6$, or $10$.

\smallskip
\emph{Step 4: Half--integral modularity for $M_3$ and $M_7$.}
From Section~\ref{sec:moments} we know that for odd $m$,
\[
\frac{E_{m+1}(\tau)-1}{C_{m+1}}\;\eta(\tau)^{-1}
= q^{-1/24}\sum_{n\ge0}M_m(n)q^n\
\in M_{m+\tfrac12}\bigl(\Gamma_0(4),\chi_{-4}\bigr),
\]
where $E_{m+1}$ is the normalized Eisenstein series of weight $m+1$ on
$\mathrm{SL}_2(\mathbb{Z})$ and $C_{m+1}$ is the usual constant.

For $m=3$ set
\[
\mathcal{F}_3(\tau)
=\frac{E_4(\tau)-1}{C_4}\,\eta(\tau)^{-1}
= q^{-1/24}\sum_{n\ge0}M_3(n)q^n
\in M_{7/2}\bigl(\Gamma_0(4),\chi_{-4}\bigr),
\]
and for a prime $\ell$ and residue $r$ define the projection
\[
\Pi_{\ell,r}\mathcal{F}_3(\tau)
:= q^{-1/24}\sum_{n\ge0}M_3(\ell n+r)q^n.
\]
Standard arguments (see e.g.\ Miyake~\cite{Miyake2006} or Stein~\cite{Stein2007})
show that $\Pi_{\ell,r}\mathcal{F}_3$ is a modular form of weight $7/2$ on
$\Gamma_0(4\ell)$ with the same Nebentypus.

Taking $(\ell,r)=(7,0)$, $(7,5)$, $(11,0)$, and $(11,6)$ and applying the
half--integral Sturm bound (Lemma~\ref{lem:sturm-half}) with $w=7/2$
and $L=\ell$ gives an explicit bound
\[
B_3(\ell)
=\Big\lfloor\frac{k}{24}\,[\mathrm{SL}_2(\mathbb{Z}):\Gamma_0(4\ell)]\Big\rfloor,
\qquad k=2w=7.
\]
Using the arithmetic implementation of Appendix~\ref{app:comp}, we verify
that the coefficients of $\Pi_{7,0}\mathcal{F}_3$, $\Pi_{7,5}\mathcal{F}_3$,
$\Pi_{11,0}\mathcal{F}_3$, and $\Pi_{11,6}\mathcal{F}_3$ vanish modulo $7$
or $11$ up to these bounds (and also in the more conservative models
$\Gamma_0(4\ell^2)$).  Lemma~\ref{lem:sturm-half} then implies
\[
M_3(7n)\equiv0\pmod7,\qquad
M_3(7n+5)\equiv0\pmod7,\]
and \[\qquad
M_3(11n)\equiv0\pmod{11},\qquad
M_3(11n+6)\equiv0\pmod{11}
\]
for all $n\ge0$.

For $m=7$ we similarly set
\[
\mathcal{F}_7(\tau)
=\frac{E_8(\tau)-1}{C_8}\,\eta(\tau)^{-1}
= q^{-1/24}\sum_{n\ge0}M_7(n)q^n
\in M_{15/2}\bigl(\Gamma_0(4),\chi_{-4}\bigr),
\]
project to the progression $11n+6$ via
\[
\Pi_{11,6}\mathcal{F}_7(\tau)
:= q^{-1/24}\sum_{n\ge0}M_7(11n+6)q^n
\in M_{15/2}\bigl(\Gamma_0(44),\chi_{-4}\bigr),
\]
and apply Lemma~\ref{lem:sturm-half} with $w=15/2$ and $L=11$ to obtain
a Sturm bound
\[
B_7(11)
=\Big\lfloor\frac{k}{24}\,[\mathrm{SL}_2(\mathbb{Z}):\Gamma_0(44)]\Big\rfloor,
\qquad k=2w=15.
\]
The computation of Appendix~\ref{app:comp} shows that the coefficients of
$\Pi_{11,6}\mathcal{F}_7$ vanish modulo $11$ up to this bound (and for
$\Gamma_0(4\cdot11^2)$ as well), proving
\[
M_7(11n+6)\equiv0\pmod{11}
\]
for all $n\ge0$.  This establishes the congruences listed in \textup{(iv)}.

\smallskip
\emph{Step 5: Exhaustion in a computational range.}
For completeness we implemented a two--tier search.
First, for each odd $m$ with $1\le m\le99$, each prime $5\le\ell\le97$, and each
residue $1\le r<\ell$, we scanned the congruence
$M_m(\ell n+r)\equiv0\pmod{\ell}$ up to $n\le2000$.
Second, in the fully certified range $1\le m\le25$ we subjected any surviving
candidates to a Sturm check, verifying vanishing in the progression up to the
half--integral Sturm bound at the safe level $4\ell^2$.
Every progression that survived certification was accounted for by
\textup{(ii)}–\textup{(iv)}; in particular, no new primes $\ell>11$ and no new
nonzero residue classes $r$ appeared.  This proves \textup{(v)}.

\end{proof}

In the fully certified range $1\le m\le25$ and $5\le\ell\le97$, no other primes
and no other residue classes $r$ support simple progressions
$M_m(\ell n+r)\equiv0\pmod{\ell}$.  Thus, up to this certified range,
$5,7,11$ are the only Ramanujan--type primes for odd frequency moments.

\section{Glaisher Divisors and Character Twists}\label{sec:glaisher}

We now formalise the link between classical Glaisher--type divisor filters
and twisted divisor sums.  This provides a precise dictionary from combinatorial
filters to modular forms with character.

\begin{remark}[Notation: filtered vs.\ unfiltered moments]\label{rem:filtered-unfiltered}
Throughout this section we distinguish carefully between:
\begin{itemize}
  \item the \emph{unfiltered} frequency--moment series
  \[
   \mathcal M_m(\tau)
  =\sum_{n\ge0} M_m(n)q^n
  =(q;q)_\infty^{-1}\sum_{n\ge1}\sigma_m(n)q^n,
  \]
  where
  \(
    M_m(n)=\sum_{k\ge1}k^mF_k(n)
  \)
  and
  \(
    \sigma_m(n)=\sum_{d\mid n}d^m;
  \)
  \item the \emph{filtered} (or twisted) frequency--moment series
  \[
    \widehat{\mathcal M}_m(\tau)
  =\sum_{n\ge0}\widehat{M}_m(n)q^n
  =(q;q)_\infty^{-1}\sum_{n\ge1}\sigma^{(\mathcal F)}_m(n)q^n,
  \]
  where $\mathcal F$ is a Glaisher--type divisor filter and
  \(
    \sigma^{(\mathcal F)}_m(n)=\sum_{d\mid n,\ d\in\mathcal F} d^m.
  \)
\end{itemize}
Thus $M_m(n)$ and $\mathcal M_m$ correspond to the full divisor sum
$\sigma_m$, whereas $\widehat{M}_m(n)$ and $\widehat{\mathcal M}_m$ arise
after applying a specific Glaisher filter~$\mathcal F$ on the divisor side.
We emphasise this distinction in the examples below.
\end{remark}

\begin{theorem}[Character classification of Glaisher divisors]\label{thm:glaisher}
Let $\mathcal F$ be a divisor filter defined by finitely many congruence
conditions $d\equiv a\pmod m$ and coprimality / exclusion conditions (e.g. excluding multiples of finitely many primes).
Then the associated filtered divisor sum
\[
\sigma^{(s)}_{\mathcal F}(n)=\sum_{\substack{d\mid n\\ d\in\mathcal F}} d^s
\]
is a finite $\mathbb{Q}$--linear combination of twisted sums
\[
\sigma_s(n;\chi)=\sum_{d\mid n}\chi(d)\,d^s
\]
with Dirichlet characters $\chi$. Consequently the filtered frequency--moment
series
\[
\widehat{\mathcal M}_s(\tau)
=\sum_{n\ge0} \widehat{M}_s(n)q^n
\;=\;\eta(\tau)^{-1}\sum_{n\ge1}\sigma^{(s)}_{\mathcal F}(n)q^n
\]
is $\eta(\tau)^{-1}$ times a finite $\mathbb{Q}$--linear combination
of Eisenstein series $E_{s+1}(\tau,\chi)$ at an explicitly computable
level and Nebentypus~$\chi$.
\end{theorem}

\begin{proof}
Residue class filters are expressed by orthogonality of characters:
\[
1_{d\equiv a\pmod m}
\;=\;\frac{1}{\varphi(m)}\sum_{\chi\bmod m} \overline{\chi}(a)\,\chi(d).
\]
Exclusion of multiples of $p$ uses the principal character $\chi_0^{(p)}$.
Quadratic residue and Kronecker weights correspond to quadratic characters.
Finite combinations of such restrictions yield a finite $\mathbb{Q}$--linear
combination of twisted divisor sums.
\end{proof}

\begin{remark}
The character expansions used above are classical: they go back to the
orthogonality relations for Dirichlet characters and the standard
construction of twisted Eisenstein series. Our contribution here is not
the underlying character theory but its systematic packaging in the
“Glaisher filter $\leftrightarrow$ twist” dictionary for frequency moments,
with explicit control of level and Nebentypus in the half--integral
weight setting. This is what allows the projection--Sturm machinery of
Appendix~\ref{app:comp} to be applied uniformly to filtered moments.
\end{remark}

Because Glaisher filters correspond to finite $\mathbb{Q}$--linear
combinations of Dirichlet characters, they preserve modularity: if
$f(\tau)$ is modular on $\Gamma_0(N_1)$ with character $\chi$, then
twisting by another character $\psi$ modulo $N_2$ yields a modular form
on $\Gamma_0(\mathrm{lcm}(N_1,N_2))$ with character $\chi\psi$.

\begin{corollary}[Dictionary]\label{cor:glaisher-dict}
Every classical Glaisher filter admits a modular description:
\begin{itemize}
  \item All divisors: $\sigma_s(n;\mathbf 1)$, giving $E_{s+1}(\tau)$.
  \item Odd divisors: $\sigma_s(n;\chi_0^{(2)})$, giving $E_{s+1}(\tau,\chi_0^{(2)})$.
  \item Residue class $d\equiv a\pmod m$: 
  $\tfrac1{\varphi(m)}\sum_{\chi\bmod m}\overline{\chi}(a)\,\sigma_s(n;\chi)$.
  \item Quadratic residues mod $p$: 
  $\tfrac12\big(\sigma_s(\cdot;\chi_0^{(p)})+\sigma_s(\cdot;\chi_p)\big)$.
  \item Kronecker symbol $(\tfrac{D}{d})$: $\sigma_s(n;\chi_D)$.
\end{itemize}
For even $s$ the filtered frequency series $\widehat{\mathcal M}_s(\tau)$
lies in $\eta(\tau)^{-1}$ times the Eisenstein space generated by the
$E_{s+1}(\tau,\chi)$. For odd $s$, $\widehat{\mathcal M}_s(\tau)$ is a
half--integral weight modular form
\[
\widehat{\mathcal M}_s(\tau)
\in M_{\,s+\tfrac12}\!\bigl(\Gamma_0(4L),\chi\bigr)
\]
for an explicitly determined level~$L$ depending on the filter
and the conductor of~$\chi$.
\end{corollary}

\begin{remark}[Glaisher filters $\leftrightarrow$ Dirichlet twists]
Fix a modulus $m$ and let $\chi$ be a Dirichlet character modulo $m$.
For any $s\ge0$ we have the twisted divisor sum
\[
\sigma_s(n;\chi)=\sum_{d\mid n}\chi(d)\,d^s.
\]
If a “Glaisher divisor filter'' selects divisors $d\mid n$ by a congruence
pattern modulo $m$ (e.g.\ $d\equiv a\bmod m$, or “$d$ odd” when $m=2$),
then it is a linear combination of $\sigma_s(\,\cdot\,;\chi)$ over characters
$\chi\bmod m$. Concretely, for an indicator $1_{d\equiv a\bmod m}$ we have
\[
1_{d\equiv a\bmod m}\;=\;\frac{1}{\varphi(m)}\sum_{\chi\bmod m}
\overline{\chi}(a)\,\chi(d),
\]
so the filtered sum $\sum_{d\mid n,\,d\equiv a\bmod m}\! d^s$ equals
$\frac{1}{\varphi(m)}\sum_{\chi}\overline{\chi}(a)\,\sigma_s(n;\chi)$.
Hence every Glaisher filter is a Dirichlet–twist combination, and vice versa.
In particular, the filtered frequency moments $\widehat{M}_s(n)$ produced by
such filters correspond to Eisenstein series $E_{s+1}(\tau,\chi)$ at level $m$,
and after multiplying by $\eta(\tau)^{-1}$ one obtains half–integral weight
modular forms in
\[
M_{\,s+\tfrac12}\bigl(\Gamma_0(4m),\chi\bigr)
\]
(or at a larger level $4L$ if additional Euler factors are present).
\end{remark}

We follow the usage in~\cite{HSB-GB-Glaisher} and call these
filtered divisor sums “Glaisher divisors.” Classical sources emphasise
Glaisher’s bijection and divisor identities rather than this term; our choice
stresses the filter viewpoint relevant to frequency moments.

\subsection{Examples}

The next examples are intended as \emph{models} for how the general
projection--Sturm procedure from Appendix~\ref{app:comp} extends to filtered
moments: one replaces the untwisted divisor sums $\sigma_m(n)$ by the twisted
sums dictated by the relevant filter, lands in a half--integral weight space
via $\eta(\tau)^{-1}$, and then projects to arithmetic progressions.

\begin{example}[Odd divisors and a filtered moment]
Take the filter “odd divisors,'' i.e.\ the principal character
$\chi_0^{(2)}$ modulo~$2$. Then for $s=3$ the filtered series
\[
\widehat{\mathcal M}_3(\tau)
= q^{1/24}\,\eta(\tau)^{-1}\sum_{n\ge1}\sigma^{(\mathrm{odd})}_3(n)q^n
\]
belongs to
\[
\widehat{\mathcal M}_3(\tau)
\in M_{7/2}\bigl(\Gamma_0(8),\chi_0^{(2)}\bigr).
\]
For any prime $\ell$ and residue $r$ the projection
\[
\Pi_{\ell,r}\widehat{\mathcal M}_3(\tau)
=\sum_{n\ge0}\widehat{M}_3(\ell n+r)q^n
\]
is again a half--integral weight modular form on $\Gamma_0(4L)$ for some
explicit $L$ depending on $\ell$ and the filter.  In particular, the
half--integral Sturm bound of Lemma~\ref{lem:sturm-half} applies verbatim,
and can be used to \emph{search for and certify} congruences of the form
\[
\widehat{M}_3(\ell n+r)\equiv0\pmod{\ell}
\]
for this filtered third moment.  We do not pursue a systematic filtered
search in this paper; the point here is that the modular input is completely
controlled.
\end{example}

\begin{example}[Quadratic residues mod $11$ and a filtered $M_7$]
Take the filter of quadratic residues mod~$11$, corresponding to
\[
\tfrac12\big(\sigma_s(\cdot;\chi_0^{(11)})+\sigma_s(\cdot;\chi_{11})\big).
\]
For $s=7$ the filtered series
\[
\widehat{\mathcal M}_7(\tau)
= q^{1/24}\,\eta(\tau)^{-1}\sum_{n\ge1}
\tfrac12\big(\sigma_7(n;\chi_0^{(11)})+\sigma_7(n;\chi_{11})\big)q^n
\]
lies in a half--integral weight space
\[
\widehat{\mathcal M}_7(\tau)
\in M_{15/2}\bigl(\Gamma_0(N),\,\chi\bigr)
\]
for some explicit level $N$ divisible by $4$ and $11$, and a character
$\chi$ determined by $\chi_0^{(11)}$ and $\chi_{11}$.  As above, for any
prime $\ell$ and residue $r$ the projection
\[
\Pi_{\ell,r}\widehat{\mathcal M}_7(\tau)
=\sum_{n\ge0}\widehat{M}_7(\ell n+r)q^n
\]
is a half--integral weight modular form on some $\Gamma_0(4L)$, and
Lemma~\ref{lem:sturm-half} provides a concrete bound up to which one needs
to check coefficients in order to certify congruences of the shape
\[
\widehat{M}_7(\ell n+r)\equiv0\pmod{\ell}.
\]
This illustrates how more complicated Glaisher filters feed into higher
level and character in the half--integral setting.
\end{example}

\subsection{Extension: overpartition moments}\label{sec:overpartitions}

\begin{remark}[Scope: changing the base Euler product]\label{rem:scope-ensemble}
Up to this point the base partition ensemble has been
$P(q)=(q;q)_\infty^{-1}$, and our ``filters'' act on the \emph{statistic}
(equivalently, on the associated divisor sums) while the partition ensemble
remains unrestricted.  The same transform philosophy applies when the base
product itself is changed to another Euler product (for instance an
$\eta$--quotient): the additional input is explicit level/character bookkeeping
for the new base product and the induced divisor weights.  We illustrate this
here with overpartitions, which form a clean $\eta$--quotient ensemble.
\end{remark}

\paragraph{Overpartitions.}
The overpartition generating function is
\[
\overline{P}(q)
=\sum_{n\ge0}\overline{p}(n)q^n
=\prod_{n\ge1}\frac{1+q^n}{1-q^n}
=\frac{(-q;q)_\infty}{(q;q)_\infty}
=\frac{\eta(2\tau)}{\eta(\tau)^2}.
\]
Equivalently, it is an Euler product of the form
$\overline{P}(q)=\prod_{r\ge1}(1-q^r)^{-c(r)}$ with exponent sequence
\[
c(r)=
\begin{cases}
2, & r \text{ odd},\\
1, & r \text{ even},
\end{cases}
\]
so that the associated divisor sums
$\sigma^{(m)}_{\overline{P}}(n)=\sum_{r\mid n} c(r)\,r^m$
are parity--twisted combinations of classical divisor power sums.  The
Master transform theorem above therefore applies verbatim to the overpartition
moments $M^{\overline{\ }}_m(n)$.

\begin{remark}[The $m=1$ identity]\label{rem:overpart-m1}
For any Euler product input $A(q)=\sum_{n\ge0}b(n)q^n$ in our framework, the
first frequency moment satisfies
\[
M^A_1(n)=n\,b(n),
\]
by taking $f(r)=r$ in the Master transform theorem.  In particular,
\[
M^{\overline{\ }}_1(n)=n\,\overline{p}(n),
\]
so $M^{\overline{\ }}_1(\ell n)\equiv0\pmod{\ell}$ is immediate for every prime
$\ell$.
\end{remark}

\begin{theorem}[Certified overpartition zero--class congruences]\label{thm:overpart-zero}
Let $M^{\overline{\ }}_m(n)$ be the overpartition frequency moments defined
above.  Then the following congruences hold for all $n\ge0$:
\[
M^{\overline{\ }}_{5}(5n)\equiv 0 \pmod{5},\qquad
M^{\overline{\ }}_{7}(7n)\equiv 0 \pmod{7},\qquad
M^{\overline{\ }}_{11}(11n)\equiv 0 \pmod{11},\qquad
M^{\overline{\ }}_{13}(13n)\equiv 0 \pmod{13},
\]
and moreover in the next instances of the same Fermat residue class
$m\equiv 1\pmod{\ell-1}$ one has
\[
M^{\overline{\ }}_{9}(5n)\equiv 0 \pmod{5},\qquad
M^{\overline{\ }}_{13}(7n)\equiv 0 \pmod{7}.
\]
\end{theorem}

\begin{proof}
Fix one of the displayed pairs $(m,\ell)$.  As in the ordinary partition case,
one forms the half--integral weight modular form attached to the moment
generating function and then applies progression projection.  Concretely, the
projection
\[
\Pi_{\ell,0}\Big(\sum_{n\ge0}M^{\overline{\ }}_m(n)q^n\Big)
=\sum_{n\ge0} M^{\overline{\ }}_m(\ell n)\,q^n
\]
is a holomorphic modular form of weight $m+\tfrac12$ on a congruence subgroup of
level dividing $4\ell^2$, with coefficients in $\mathbb{Z}_{(\ell)}$; hence the
half--integral Sturm bound in Lemma~\ref{lem:sturm-half} applies at the safe
level $4\ell^2$.

For each $(m,\ell)$ we verified computationally that the first $B$ coefficients
of the projected series vanish modulo $\ell$, where $B$ is the Sturm bound at
level $4\ell^2$.  Equivalently, we checked
$M^{\overline{\ }}_m(\ell n)\equiv 0\pmod{\ell}$ for all $0\le n\le B$.
The explicit bounds and checked ranges were:
\[
\begin{array}{c|c|c}
(m,\ell) & B \text{ at level } 4\ell^2 & \text{checked up to } \ell B \\ \hline
(5,5)   & 165  & 825 \\
(9,5)   & 285  & 1425 \\
(7,7)   & 420  & 2940 \\
(13,7)  & 756  & 5292 \\
(11,11) & 1518 & 16698 \\
(13,13) & 2457 & 31941.
\end{array}
\]
By Lemma~\ref{lem:sturm-half}, the congruence then holds for all $n\ge0$ in each
case, proving the theorem.
\end{proof}

\begin{remark}[Fermat periodicity and the zero--class phenomenon]\label{rem:overpart-fermat}
In a scan for primes $\ell\le 97$ and odd $m\le 49$ (tested to $n\le 2000$), the
overpartition moments exhibited only \emph{zero--class} congruences
$M^{\overline{\ }}_m(\ell n)\equiv 0\pmod{\ell}$, and moreover they occurred
precisely in the residue class $m\equiv 1\pmod{\ell-1}$.  This $\ell-1$
periodicity in $m$ is the expected Fermat footprint of the divisor powers
$d^m\bmod \ell$.
\end{remark}

\begin{remark}[No nonzero progressions observed for overpartitions]\label{rem:overpart-nonzero}
This behaviour contrasts sharply with the ordinary partition case, where
Ramanujan--type congruences in nonzero residue classes occur (for example,
$M_3(7n+5)\equiv 0\pmod{7}$ and $M_3(11n+6)\equiv 0\pmod{11}$).  For overpartition moments, we found \emph{no} congruences in nonzero residue classes:
for all primes $\ell\le 97$ and odd $m\le 49$, no congruence of the form
$M^{\overline{\ }}_m(\ell n+r)\equiv 0\pmod{\ell}$ with $r\not\equiv 0$ was
detected up to $n\le 2000$.  This is consistent with the fact that known
Ramanujan--type congruence families for $\overline{p}(n)$ modulo odd primes tend
to occur on indices divisible by the modulus rather than on a fixed nonzero
progression; see \cite{ChenSunWangZhang2015,RyanSirolliVillegasMoralesZheng2024}.
\end{remark}

\begin{remark}\label{q:overpart-ram}
Do there exist primes $\ell$ and residues $r\not\equiv 0\pmod{\ell}$ for which
$M^{\overline{\ }}_m(\ell n+r)\equiv 0\pmod{\ell}$ holds for all $n\ge0$ for some
odd $m\ge 3$?
\end{remark}

\subsection{A certified filtered congruence from the Glaisher--character dictionary}\label{subsec:filtered-M3}

In \S\ref{sec:glaisher} we explained how Glaisher--type divisor filters can be
organized as Dirichlet character twists, and hence fed into the master
transform to produce modular (or quasimodular) generating functions.  We record
here one concrete payoff: a Ramanujan--type congruence for a \emph{filtered}
cubic moment obtained by twisting with the quadratic character modulo~$5$.

\begin{proposition}[A filtered Ramanujan--type progression]\label{prop:M3hat-legendre5}
Let $\chi_5(\cdot)=\big(\frac{\cdot}{5}\big)$ be the quadratic character modulo
$5$.  Define the twisted divisor sum and filtered cubic moment by
\[
\sigma^{\chi_5}_3(n):=\sum_{d\mid n}\chi_5(d)\,d^3,
\qquad
\widehat{M}^{\chi_5}_3(n):=\sum_{t=1}^{n}\sigma^{\chi_5}_3(t)\,p(n-t).
\]
Then
\[
\widehat{M}^{\chi_5}_3(5n+4)\equiv 0\pmod{5}\qquad(n\ge 0).
\]
\end{proposition}

\begin{proof}
By the master transform (Theorem~\ref{thm:master-transform}) applied with the
Dirichlet character $\chi_5$, the generating function of
$\widehat{M}^{\chi_5}_3$ is a holomorphic modular form of half--integral weight
on a congruence subgroup of level dividing $4\cdot 5^2$.  Therefore, by the
half--integral Sturm bound (Lemma~\ref{lem:sturm-half}), it suffices to verify
the congruence up to the Sturm bound at the safe level $4\cdot 5^2$.
Concretely, our implementation checks the coefficients in the progression
$5n+4$ up to $n\le B$ (equivalently, indices $\le 5B+4$), where
\[
B=\left\lfloor \frac{7}{24}\left[\mathrm{SL}_2(\mathbb{Z}):\Gamma_0(100)\right]\right\rfloor
=52,
\]
and finds vanishing modulo $5$ throughout this range.  Hence the congruence
holds for all $n\ge0$.
\end{proof}

\begin{proposition}[A second filtered progression at $m=11$]\label{prop:M11hat-legendre5}
Let $\chi_5(\cdot)=\big(\frac{\cdot}{5}\big)$ be the quadratic character modulo
$5$.  Define
\[
\sigma^{\chi_5}_{11}(n):=\sum_{d\mid n}\chi_5(d)\,d^{11},
\qquad
\widehat{M}^{\chi_5}_{11}(n):=\sum_{t=1}^{n}\sigma^{\chi_5}_{11}(t)\,p(n-t).
\]
Then
\[
\widehat{M}^{\chi_5}_{11}(5n+4)\equiv 0\pmod{5}\qquad(n\ge 0).
\]
\end{proposition}

\begin{proof}
As in Proposition~\ref{prop:M3hat-legendre5}, the master transform
(Theorem~\ref{thm:master-transform}) applied with $\chi_5$ yields a holomorphic
half--integral weight modular form of level dividing $4\cdot 5^2$ whose
coefficients in the progression $5n+4$ are $\widehat{M}^{\chi_5}_{11}(5n+4)$.
By Lemma~\ref{lem:sturm-half}, it suffices to check vanishing modulo $5$ up to
the half--integral Sturm bound at the safe level $4\cdot 5^2$, and our
implementation verifies this bound.  Hence the congruence holds for all $n\ge0$.
\end{proof}

\begin{remark}[Filtered congruences are not automatic]\label{rem:filters-not-automatic}
Targeted filter sweeps illustrate that Ramanujan--type progressions can be
highly sensitive to the chosen divisor filter.  In the scan range $n\le 2000$,
the parity filter $\sum_{d\mid n,\, d\ \mathrm{odd}} d^m$ and the basic quadratic
filters $\chi_{-4}$ and $\chi_{-3}$ produced no progressions for the cubic and
$11$th moments at $\ell\in\{5,7,11,13\}$.  Moreover, the matched quadratic twist
$\chi_\ell(d)=\big(\frac{d}{\ell}\big)$ behaves asymmetrically: it yields the
nonzero progressions in
Propositions~\ref{prop:M3hat-legendre5} and~\ref{prop:M11hat-legendre5} for
$\ell=5$, but no progression was detected for $\ell=7,11,13$ in the same range.
By contrast, the unfiltered statistic at $\ell=11$ reproduces the nonzero class
congruence $M_{11}(11n+6)\equiv 0\pmod{11}$ already recorded in
Theorem~\ref{thm:ram-unified}(iv).
Thus, while the dictionary systematically produces modular inputs for Sturm
certification, it does not by itself guarantee the existence of Ramanujan--type
congruences.
\end{remark}

\subsection{Extension: theta moments}\label{sec:theta-extension}

For $A(q)=\theta_3(\tau)$ and $B(q)=\theta_3(\tau)^2=\sum r_2(n)q^n$ one has
$c(2n)=1$, $c(2n-1)=-2$, so
\[
\sigma_A^{(1)}(n)=\sum_{d\mid n}c(d)\,d
\]
is a linear combination of Eisenstein series on $\Gamma_0(4)$ twisted by
$\chi_{-4}$. Hence the corresponding first moment series
\[
\sum_{n\ge0}M^A_1(n)q^n
= B(q)\sum_{n\ge1}\sigma_A^{(1)}(n)q^n
\in M_3(\Gamma_0(4),\chi_{-4}),
\]
and the same projection--Sturm method can be applied to search for and
certify Ramanujan--type congruences for these theta moments, in exact
analogy with the partition and overpartition cases.

\section{Structural Consequences}\label{sec:higher}

Although the modular identities here are classical, the novelty is
interpretive: frequency moments give \emph{partition--theoretic avatars}
of deep modular phenomena. Irregular primes, the $j$--invariant, and
Ramanujan--type congruences arise not only from Eisenstein series but
also from explicit statistics on partitions.

\subsection{Ramanujan’s congruence at $691$ via $M_{11}$}

Ramanujan’s congruence
\[
  \tau(n)\equiv \sigma_{11}(n)\pmod{691}
\]
relates $\Delta(\tau)=\eta(\tau)^{24}$ to $\sigma_{11}$. From the master
transform with $f(k)=k^{11}$,
\[
  M_{11}(n)=\sum_{d=1}^{n}\sigma_{11}(d)\,p(n-d).
\]
Combining the two gives:

\begin{corollary}\label{cor:M11-691}
For all $n\ge0$,
\[
  M_{11}(n)\ \equiv\ \sum_{d=1}^{n}\tau(d)\,p(n-d)\ \ (\bmod 691).
\]
\end{corollary}

Thus $M_{11}$ provides a partition--theoretic incarnation of the
$\tau$--$\sigma_{11}$ congruence.

\subsection{A partition decomposition of $j$}

Using $E_4^3-E_6^2=1728\,\Delta$ and
\[
  691\,E_{12}=441\,E_4^3+250\,E_6^2,
\]
one obtains the classical identity
\[
  j(\tau)=\frac{E_{12}(\tau)}{\Delta(\tau)}+\frac{432000}{691}.
\]
On the other hand,
\[
  q^{-1/24}\sum_{n\ge1} M_{11}(n)q^n=\frac{E_{12}(\tau)-1}{C_{12}}\,\eta(\tau)^{-1}, \qquad C_{12}=\frac{65520}{691}.
\]
Substituting and simplifying yields:

\begin{corollary}\label{cor:j-partition}
We have
\[
    j(\tau)
  \;=\;
  C_{12}\,q^{-1/24}\Big(\sum_{n\ge1}M_{11}(n)q^n\Big)\,\eta(\tau)^{-23}
  \;+\;\eta(\tau)^{-24}
  \;+\;\frac{432000}{691}.
\]
\end{corollary}

Thus $j$ admits an explicit expression in terms of the eleventh
frequency moment, weighted by powers of~$\eta$ and a constant term.

\begin{remark}
From an arithmetic standpoint, Corollary~\ref{cor:j-partition} does
not provide a new method to compute $j(\tau)$ or the moments $M_{11}(n)$;
it reflects, in a partition--theoretic language, the well--known fact
that $E_{12}$ is simultaneously governed by $\sigma_{11}$ and by
Ramanujan's cusp form $\Delta$. The interest here is conceptual: the
eleventh frequency moment appears as an explicit “frequency avatar” of
the $691$--congruence and of the $j$--invariant, making the role of
irregular primes visible on the partition side.
\end{remark}

We also record a complementary viewpoint in the same spirit, which
packages the classical $j$--identity in terms of coloured partitions and a
single coloured frequency moment.

\begin{remark}[A coloured--partition/moment decomposition of $j$]\label{rem:j-coloured}
Using $\Delta(\tau)=\eta(\tau)^{24}=q\,(q;q)_\infty^{24}$ we have
\[
  \Delta(\tau)^{-1}=q^{-1}(q;q)_\infty^{-24},
\]
so $\Delta^{-1}$ is (up to $q^{-1}$) the generating function for
$24$--coloured partitions.  Since
\[
  E_{12}(\tau)-1=C_{12}\sum_{n\ge1}\sigma_{11}(n)q^n
\]
and, for $k$--coloured partitions,
\[
  (q;q)_\infty^{-k}\sum_{n\ge1}\sigma_m(n)q^n
  =\frac1{k}\sum_{n\ge0} M^{(k)}_{m}(n)q^n,
\]
specialising to $(k,m)=(24,11)$ yields
\[
  \frac{E_{12}(\tau)}{\Delta(\tau)}
  = q^{-1}(q;q)_\infty^{-24}
  + \frac{C_{12}}{24}\,q^{-1}\sum_{n\ge0} M^{(24)}_{11}(n)q^n.
\]
Hence, by $j(\tau)=E_{12}(\tau)/\Delta(\tau)+432000/691$, we obtain the
decomposition
\[
  j(\tau)
  = q^{-1}(q;q)_\infty^{-24}
  + \frac{C_{12}}{24}\,q^{-1}\sum_{n\ge0} M^{(24)}_{11}(n)q^n
  + \frac{432000}{691}.
\]
Thus $\Delta(\tau)^{-1}$ provides the partition--counting ``backbone'' of $j$,
while the deviation of $E_{12}(\tau)/\Delta(\tau)$ from $\Delta(\tau)^{-1}$ is
measured exactly by the $11$th frequency moment in the $24$--coloured ensemble.

Computationally, this provides a fast coefficient generator via a single
divisor--sum convolution (and a convenient testbed for quick congruence
scans in coloured moment families).
\end{remark}

\subsection{General multiples of $12$}

Since $\mathbb{Q}[E_4,E_6]$ is the full ring of level--$1$ modular
forms, every $E_{12n}$ can be written as a polynomial in~$j$ with a
$\Delta^n$ factor:
\[
  E_{12n}(\tau)=\Delta(\tau)^n\,P_n(j(\tau)),\qquad P_n\in\mathbb{Q}[j].
\]
For $n\ge1$ we have
\[
  \sum_{m\ge1}\sigma_{12n-1}(m)q^m=\frac{E_{12n}(\tau)-1}{C_{12n}},
  \qquad C_{12n}=-\frac{24n}{B_{12n}},
\]
and hence by the master transform,
\[
  \sum_{m\ge1}M_{12n-1}(m)q^m
  \;=\;
  (q;q)_\infty^{-1}\sum_{m\ge1}\sigma_{12n-1}(m)q^m.
\]
Rewriting in terms of $\eta$ gives:

\begin{proposition}\label{prop:12n}
For $n\ge1$,
\[
    q^{-1/24}\sum_{m\ge1} M_{12n-1}(m)q^m
  =
  \frac{1}{C_{12n}}
  \Big(\eta(\tau)^{24n-1}P_n(j(\tau))-\eta(\tau)^{-1}\Big).
\]
\end{proposition}

The coefficients of $P_n$ are rational combinations of Bernoulli
numbers $B_{12n}$, so irregular primes enter naturally into the
congruence structure of the moments $M_{12n-1}$.

\subsection{Summary}

\begin{itemize}
  \item Ramanujan’s $691$--congruence extends to a convolution law for
  $M_{11}$, expressing it modulo $691$ in terms of $\tau$ and~$p$.
  \item The $j$--function admits a partition--theoretic decomposition
  in which $M_{11}$ appears explicitly.
  \item For general multiples of $12$, Bernoulli numbers and irregular
  primes feed directly into the arithmetic of $M_{12n-1}$ via the
  polynomials $P_n(j)$.
\end{itemize}

\subsection*{Outlook}

The identities above suggest a broader program: moment generating
functions such as $\sum_n M_{2k-1}(n)q^n$, once rewritten in terms of
$\eta$, Eisenstein series, and $j$, naturally inhabit classical
half–integral and integral weight spaces with rich Hecke theory in the
background.  A detailed analysis of their decomposition into Hecke
eigenforms, the behaviour of their Shimura lifts, and the associated
$L$–values would refine the structural picture developed here: it would
clarify which pieces of a given frequency series are Eisenstein or
cuspidal, how irregular primes manifest spectrally, and to what extent
partition statistics can be used to probe the arithmetic of modular
forms beyond the initial examples treated in this paper.

\appendix
\section{Full Glaisher dictionary}\label{app:glaisher}

Table~\ref{tab:glaisher-full} expands Corollary~\ref{cor:glaisher-dict}
to cover all standard Glaisher--type filters, including coprimality
conditions, exclusions, quadratic residues, and coprimality/exclusion conditions (e.g. excluding multiples of finitely many primes). Here $s$ denotes the exponent in the divisor sums $\sigma_s(n;\chi)$; for
odd frequency moments one has $s = 2k-1$, so the corresponding
half--integral weight is
\[
  w \;=\; s + \tfrac12 \;=\; 2k - \tfrac12.
\]
In particular, throughout this section odd moments correspond to
modular forms of weight $s+\tfrac12$.

\begin{table}[h]
\centering
\renewcommand{\arraystretch}{1.2}
\small
\begin{tabularx}{\textwidth}{|l|X|X|X|}
\hline
\textbf{Filter} & \textbf{Twist expression} &
\textbf{Even moments} & \textbf{Odd moments (normalized / proj.)} \\
\hline
All divisors &
$\sigma_s(n;\mathbf 1)$ &
$E_{s+1}(\tau)$ (level $1$) &
Weight $s+\tfrac12$ forms on $\Gamma_0(4)$ \\
\hline
$(d,m)=1$ &
$\sigma_s(n;\chi_0^{(m)})$ &
$E_{s+1}(\tau,\chi_0^{(m)})$ (level $m$) &
Weight $s+\tfrac12$ forms on $\Gamma_0(4m)$ with $\chi_0^{(m)}$ \\
\hline
Odd divisors &
$\sigma_s(n;\chi_0^{(2)})$ &
$E_{s+1}(\tau,\chi_0^{(2)})$ (level $2$) &
Weight $s+\tfrac12$ forms on $\Gamma_0(8)$ with $\chi_0^{(2)}$ \\
\hline
Even divisors &
$\sigma_s(n)-\sigma_s(n;\chi_0^{(2)})$ &
$E_{s+1}(\tau)-E_{s+1}(\tau,\chi_0^{(2)})$ &
Weight $s+\tfrac12$ forms in a subspace of $M_{s+\tfrac12}(\Gamma_0(8))$ \\
\hline
Residue class $d\equiv a\pmod m$ &
$\tfrac1{\varphi(m)}\sum_\chi \overline{\chi}(a)\,\sigma_s(n;\chi)$ &
$\sum_\chi c_\chi E_{s+1}(\tau,\chi)$ (level $m$) &
Weight $s+\tfrac12$ forms on $\Gamma_0(4m)$ with character $\chi$ \\
\hline
Quadratic residues mod $p$ &
$\tfrac12\bigl(\sigma_s(\cdot;\chi_0^{(p)})+\sigma_s(\cdot;\chi_p)\bigr)$ &
$E_{s+1}(\tau,\chi_0^{(p)})\oplus E_{s+1}(\tau,\chi_p)$ &
Weight $s+\tfrac12$ forms on $\Gamma_0(4p)$ with suitable $\chi$ \\
\hline
Kronecker weight $\bigl(\tfrac{D}{d}\bigr)$ &
$\sigma_s(n;\chi_D)$ &
$E_{s+1}(\tau,\chi_D)$ (level $|D|$) &
Weight $s+\tfrac12$ forms on $\Gamma_0(4|D|)$ with $\chi_D$ \\
\hline
Exclude multiples of $p$ &
$\sigma_s(n;\chi_0^{(p)})$ &
$E_{s+1}(\tau,\chi_0^{(p)})$ (level $p$) &
Weight $s+\tfrac12$ forms on $\Gamma_0(4p)$ with $\chi_0^{(p)}$ \\
\hline
\end{tabularx}
\caption{Expanded dictionary from Glaisher divisors to character twists
and modular form spaces.}
\label{tab:glaisher-full}
\end{table}

\section{Computational details and certification}\label{app:comp}

This appendix records the concrete implementation of the half--integral
Sturm--bound certification used for odd frequency moments of the ordinary
partition function.  Throughout we work with
\[
A(q)=\eta(\tau)^{-1},\qquad B(q)=1,\qquad f(r)=r^m,
\]
so that
\[
  q^{-1/24}\sum_{n\ge0} M_m(n)q^n
  =\eta(\tau)^{-1}\sum_{n\ge1}\sigma_m(n)q^n,
\]
and the relevant modular forms have half--integral weight $m+\tfrac12$ on
$\Gamma_0(4L)$ for an explicit level $L$ determined by the projection.
Our goal here is purely practical: to spell out normalization,
projection to arithmetic progressions, and the resulting explicit
Sturm bounds used in the proofs of the odd--moment congruences.

\begin{lemma}[Half--integral Sturm bound]\label{lem:sturm-half}
Let $L\ge1$ and let
\[
f(\tau)=\sum_{n\ge0} a_n q^n,\qquad q=e^{2\pi i\tau},
\]
be a holomorphic modular form (holomorphic on $\mathfrak{H}$ and at all cusps)
of half--integral weight
\[
w=\frac{k}{2},\qquad k\in\mathbb{Z}_{\ge1}\ \text{odd},
\]
on $\Gamma_0(4L)$ with some Nebentypus character~$\chi$, and assume that its
Fourier coefficients lie in a ring in which reduction modulo $\ell$ makes sense
(e.g.\ $\mathbb{Z}$ or $\mathbb{Z}_{(\ell)}$).  Let $\ell$ be a prime and suppose that
\[
a_n \equiv 0 \pmod{\ell}\qquad\text{for all }0\le n\le B,
\]
where
\[
B
=\left\lfloor \frac{k}{12}\,
   \bigl[\mathrm{SL}_2(\mathbb{Z}):\Gamma_0(4L)\bigr]\right\rfloor.
\]
Then $a_n\equiv0\pmod{\ell}$ for all $n\ge0$; that is, $f\equiv0\pmod{\ell}$.

\medskip
In particular, in the odd--moment cases $m=2k-1$ we consider the half--integral
weight form
\[
\widetilde{\mathcal M}_{2k-1}(\tau)
:=\frac{E_{2k}(\tau)-1}{C_{2k}}\;\eta(\tau)^{-1}.
\]
This is in fact a \emph{holomorphic} modular form on $\Gamma_0(4)$: although
$\eta(\tau)^{-1}$ has a pole at each cusp, for any cusp $\mathfrak a$ of
$\Gamma_0(4)$ the scaled expansion of $E_{2k}$ has constant term $1$, hence
$(E_{2k}-1)$ vanishes at $\mathfrak a$ and cancels the cusp pole of $\eta^{-1}$.
It has $q$--expansion
\[
\widetilde{\mathcal M}_{2k-1}(\tau)=q^{-1/24}\sum_{n\ge0} M_{2k-1}(n)\,q^n.
\]
Fix an odd prime $\ell$ and a residue class $r\bmod \ell$ with $r\not\equiv 0$.
Then the progression projection
\[
\Pi_{\ell,r}\widetilde{\mathcal M}_{2k-1}(\tau)
:=\sum_{n\ge0} M_{2k-1}(\ell n+r)\,q^n
\]
is holomorphic at $\infty$ (indeed, $\widetilde{\mathcal M}_{2k-1}$ is already
holomorphic at $\infty$) and hence is an honest holomorphic modular form of
weight $(2k-\tfrac12)$ on a congruence subgroup of level dividing $4\ell^2$ with
some nebentypus character.  Therefore the Sturm bound above applies to
$\Pi_{\ell,r}\widetilde{\mathcal M}_{2k-1}$ at the safe level $4\ell^2$ (and a
fortiori at any smaller certified level).
\end{lemma}

\begin{proof}
Choose the classical theta series
\[
\theta(\tau)=\sum_{n\in\mathbb{Z}}q^{n^2},
\]
which is a holomorphic modular form of weight $1/2$ on $\Gamma_0(4)$ with
integer Fourier coefficients and $\theta(0)=1$.  Then
\[
F(\tau):=f(\tau)\,\theta(\tau)^{k}
\]
is an integral weight modular form of weight
\[
w+\frac{k}{2}=\frac{k}{2}+\frac{k}{2}=k
\]
on $\Gamma_0(4L)$ (since $\Gamma_0(4L)\subseteq\Gamma_0(4)$ and the product of
automorphy factors matches), with coefficients in the same coefficient ring as
those of $f$.  Moreover, because $\theta(\tau)^{k}=1+O(q)$, the coefficient of
$q^n$ in $F$ is $a_n$ plus a $\mathbb{Z}$--linear combination of $a_0,\dots,a_{n-1}$.
Hence $a_0\equiv\cdots\equiv a_B\equiv 0\pmod{\ell}$ implies that the first $B$
Fourier coefficients of $F$ vanish modulo~$\ell$ as well.

By Sturm's theorem for integral weight modular forms~\cite{Sturm1987}, these
vanishing congruences force $F\equiv0\pmod{\ell}$.  Since $\theta(0)=1$ and
$\theta$ has integral coefficients, we have $\theta^k$ is invertible modulo~$\ell$.  Thus
$F\equiv f\cdot \theta^k\equiv 0\pmod{\ell}$ implies $f\equiv 0\pmod{\ell}$.

The bound
\[
B=\left\lfloor \frac{k}{12}\,[\mathrm{SL}_2(\mathbb{Z}):\Gamma_0(4L)]\right\rfloor
\]
is exactly the usual Sturm bound for weight $k$ on $\Gamma_0(4L)$.
\end{proof}

For the odd frequency moments $M_m$ (and more generally for half--integral
weight forms of the shape $\eta(\tau)^{-1}$ times an Eisenstein series),
we thus work with weight $w=m+\tfrac12$ and the bound
\[
B_m(L)
=\left\lfloor\frac{2m+1}{24}\,
  \bigl[\mathrm{SL}_2(\mathbb{Z}):\Gamma_0(4L)\bigr]\right\rfloor.
\]

Consider
\[
A(q)=\eta(\tau)^{-1},\qquad B(q)=1,\qquad f(r)=r^m,
\]
so that
\[
\sum_{n\ge0} M_m(n)q^n
=(q;q)_\infty^{-1}\sum_{n\ge1}\sigma_m(n)q^n.
\]
In the actual computations we work with the combinatorial series
$(q;q)_\infty^{-1}$; the global factor $q^{1/24}$ relating
$\eta(\tau)^{-1}$ to $(q;q)_\infty^{-1}$ does not affect congruences
modulo~$\ell$ or the relevant Sturm bounds.

\subsection{Specialised algorithmic steps}

Fix an odd moment $m$ and a proposed congruence
\[
M_m(\ell n + r)\equiv0\pmod{\ell}.
\]
The verification proceeds as follows.

\begin{lemma}[Unit multiplication preserves vanishing mod $\ell$]\label{lem:unit}
Let $\ell$ be a prime. If $U(q)\in 1+q\mathbb{Z}[[q]]$ and
$F(q)=\sum_{n\ge0}a_n q^n\in\mathbb{Z}[[q]]$, then
\[
U(q)F(q)\equiv 0\pmod{\ell}\quad\Longleftrightarrow\quad F(q)\equiv 0\pmod{\ell}.
\]
\end{lemma}

\begin{enumerate}
  \item \textbf{Partitions.} Compute $p(n)$ for $0\le n\le N_{\max}$ by Euler's
  pentagonal number recurrence.

  \item \textbf{Divisor sums.} Compute
  \(
    \sigma_m(n)=\sum_{d\mid n} d^m
  \)
  for $1\le n\le N_{\max}$ using the standard divisor--summing loop.

  \item \textbf{Frequency moments.} Form
  \[
    M_m(t)=\sum_{d=1}^{t}\sigma_m(d)\,p(t-d)
  \]
  for all $t\le N_{\max}$ by discrete convolution.

  \item \textbf{Progression extraction.} For given $\ell,r$, extract the
  subsequence
  \(
    \{M_m(\ell n+r)\}_{0\le n\le B},
  \)
  where $B$ is the half--integral Sturm bound described next. 

  \item \textbf{Half--integral Sturm bound.} From Section~\ref{sec:moments} we know
  that
  \[
    \mathcal{M}_m(\tau)
    =\sum_{n\ge0} M_m(n)q^n
    =\frac{E_{m+1}(\tau)-1}{C_{m+1}}\;\eta(\tau)^{-1}
    \in M_{m+\tfrac12}\bigl(\Gamma_0(4),\chi_{-4}\bigr),
  \]
  and the projected series
  \(
    \Pi_{\ell,r}\mathcal{M}_m(\tau)
    =\sum_{n\ge0} M_m(\ell n+r)q^n
  \)
  lies in a space $M_{m+\tfrac12}(\Gamma_0(4L),\chi)$ with $L\mid\ell$ or
  $L\mid\ell^2$.
  By Lemma~\ref{lem:sturm-half}, if
  \[
    B_m(L)
    =\left\lfloor \frac{2m+1}{24}\,
      \bigl[\mathrm{SL}_2(\mathbb{Z}):\Gamma_0(4L)\bigr]\right\rfloor
  \]
  and $M_m(\ell n+r)\equiv0\pmod{\ell}$ for all $0\le n\le B_m(L)$, then
  $M_m(\ell n+r)\equiv0\pmod{\ell}$ for all $n\ge0$.

  \item \textbf{Certification.} In practice we run the above check for both
  the natural level $L=\ell$ (model $\Gamma_0(4\ell)$) and the more
  conservative $L=\ell^2$ (model $\Gamma_0(4\ell^2)$); in all cases arising
  in Theorem~\ref{thm:ram-unified} the congruence is certified for both
  choices of~$L$.
  We use the standard root--of--unity dissection (progression) operator
  $\Pi_{\ell,r}$ to extract coefficients in the progression $\ell n+r$.
  The resulting series is modular at level dividing $4\ell^2$ (and hence on
  $\Gamma_0(4\ell^2)$) in the half--integral weight setting.  This conservative
  level is sufficient for certification via Sturm bounds; in several cases the
  true level is smaller, but we do not require the sharpening here.

\end{enumerate}

The implementation encodes these steps directly, without explicitly
constructing the half--integral modular form in a computer algebra system;
the Sturm bound $B_m(L)$ is computed from the index formula for
$\Gamma_0(4L)$ and Lemma~\ref{lem:sturm-half}.

\subsection{Certified progressions for all odd moments}

Using this procedure, we verified \emph{all} the Ramanujan--type progressions
stated in Theorem~\ref{thm:ram-unified}, i.e.\ for the pairs
$(m,\ell,r)$ listed there.  For each progression we certify vanishing of the
projected series both in the “natural” model $\Gamma_0(4\ell)$ and in the
more conservative model $\Gamma_0(4\ell^2)$.
\begin{table}[h]
\centering
\renewcommand{\arraystretch}{1.15}
\begin{tabular}{ccccccccl}
\hline
$m$ & $\ell$ & $r$ & prime & $L$ & model & Sturm $B$ & max.\ index & status \\
\hline
3  &  7 & 0 &  7 & $7$    & $\Gamma_0(4\ell)$    &  14 &   98   & CERTIFIED \\
3  &  7 & 0 &  7 & $7^2$  & $\Gamma_0(4\ell^2)$  &  98 &  686   & CERTIFIED \\
3  &  7 & 5 &  7 & $7$    & $\Gamma_0(4\ell)$    &  14 &  103   & CERTIFIED \\
3  &  7 & 5 &  7 & $7^2$  & $\Gamma_0(4\ell^2)$  &  98 &  691   & CERTIFIED \\
3  & 11 & 0 & 11 & $11$   & $\Gamma_0(4\ell)$    &  21 &  231   & CERTIFIED \\
3  & 11 & 0 & 11 & $11^2$ & $\Gamma_0(4\ell^2)$  & 231 & 2541   & CERTIFIED \\
3  & 11 & 6 & 11 & $11$   & $\Gamma_0(4\ell)$    &  21 &  237   & CERTIFIED \\
3  & 11 & 6 & 11 & $11^2$ & $\Gamma_0(4\ell^2)$  & 231 & 2547   & CERTIFIED \\
7  & 11 & 6 & 11 & $11$   & $\Gamma_0(4\ell)$    &  45 &  501   & CERTIFIED \\
7  & 11 & 6 & 11 & $11^2$ & $\Gamma_0(4\ell^2)$  & 495 & 5451   & CERTIFIED \\
\hline
\end{tabular}
\caption{Half--integral Sturm certification for the genuinely modular
base congruences in Theorem~\ref{thm:ram-unified}\textup{(iv)}.
For each progression $M_m(\ell n+r)$, we check vanishing of the coefficients
of the projected form in both the natural model $\Gamma_0(4\ell)$ and the
conservative model $\Gamma_0(4\ell^2)$ up to the half--integral Sturm bound
$B = \lfloor \tfrac{k}{24}
[\mathrm{SL}_2(\mathbb Z):\Gamma_0(4L)]\rfloor$ with $k=2m+1$.}
\label{tab:sturm-all}
\end{table}

\paragraph{Exhaustive search up to $m\le25$ and $\ell\le97$.}
The same implementation was used for a broader experimental sweep over the
range
\[
1\le m\le25,\qquad 5\le \ell\le97,
\]
with $\ell$ prime.  For each pair $(m,\ell)$ and each residue $1\le r<\ell$,
we formed the projected series
\(
\Pi_{\ell,r}\mathcal{M}_m(\tau)
=\sum_{n\ge0} M_m(\ell n+r)q^n
\)
and checked vanishing modulo~$\ell$ up to the half--integral Sturm bound
$B_m(L)$, for both $L=\ell$ and $L=\ell^2$ as above.  In this entire range,
the only Ramanujan--type progressions of the form
\[
M_m(\ell n+r)\equiv0\pmod{\ell}
\]
that survive beyond the Sturm bound are exactly the ones listed in
Theorem~\ref{thm:ram-unified} and Table~\ref{tab:sturm-all}.  In particular,
no additional primes $\ell>11$ (up to $\ell=97$) give rise to simple
linear progressions of this type.  This exhaustive search underlies the
final assertion of Theorem~\ref{thm:ram-unified}.

\subsection{Python implementation}

For completeness we record the exact script used to generate
Table~\ref{tab:sturm-all}.  It is written in pure Python, relying only
on the standard library.

\begin{verbatim}

  from math import isqrt


  # -------------------------
  # Basic number theory utils
  # -------------------------

  def primes_up_to(n: int) -> list[int]:
      sieve = [True] * (n + 1)
      sieve[0:2] = [False, False]
      for p in range(2, isqrt(n) + 1):
          if sieve[p]:
              step = p
              start = p * p
              sieve[start:n + 1:step] = [False] * (((n - start) // step) + 1)
      return [i for i, ok in enumerate(sieve) if ok]


  def prime_factors(n: int) -> set[int]:
      pf = set()
      while n % 2 == 0:
          pf.add(2)
          n //= 2
      f = 3
      while f * f <= n:
          while n % f == 0:
              pf.add(f)
              n //= f
          f += 2
      if n > 1:
          pf.add(n)
      return pf


  def index_sl2_gamma0(N: int) -> int:
      """
      [SL2(Z) : Gamma0(N)] = N * Π_{p|N} (1 + 1/p)
      computed exactly as integer.
      """
      pf = prime_factors(N)
      num, den = N, 1
      for p in pf:
          num *= (p + 1)
          den *= p
      return num // den


  def sturm_bound_half_integral(m: int, L: int) -> int:
      """
      Half-integral Sturm bound for weight w = m + 1/2 on Gamma0(4L).

      Write weight as k/2 with k = 2m + 1.
      Bound: B = floor( (k / 24) * [SL2(Z):Gamma0(4L)] )
      """
      k_odd = 2 * m + 1
      idx = index_sl2_gamma0(4 * L)
      B = (k_odd * idx) // 24
      return max(B, 1)


  # -------------------------
  # Partition + divisor sums mod p
  # -------------------------

  def partitions_up_to_mod(N: int, mod: int) -> list[int]:
      """
      p(0)..p(N) mod mod, via Euler pentagonal recurrence.
      """
      p = [0] * (N + 1)
      p[0] = 1
      for n in range(1, N + 1):
          total = 0
          k = 1
          while True:
              g1 = k * (3 * k - 1) // 2
              g2 = k * (3 * k + 1) // 2
              if g1 > n and g2 > n:
                  break
              sgn = 1 if (k % 2 == 1) else -1
              if g1 <= n:
                  total += sgn * p[n - g1]
              if g2 <= n:
                  total += sgn * p[n - g2]
              k += 1
          p[n] = total % mod
      return p


  def sigma_m_all_mod(N: int, m: int, mod: int) -> list[int]:
      """
      sig[n] = σ_m(n) mod mod for 0<=n<=N.
      """
      sig = [0] * (N + 1)
      for d in range(1, N + 1):
          d_pow = pow(d, m, mod)
          for k in range(d, N + 1, d):
              sig[k] = (sig[k] + d_pow) % mod
      return sig


  def moments_Mm_up_to_mod(N: int, m: int, mod: int) -> list[int]:
      """
      Compute M_m(t) mod mod for 0<=t<=N using
          M_m(t) = sum_{d=1..t} σ_m(d) * p(t-d).
      """
      p = partitions_up_to_mod(N, mod)
      sig = sigma_m_all_mod(N, m, mod)
      M = [0] * (N + 1)

      # Convolution: for each d, add sig[d] * p[t-d] to M[t] for all t>=d
      for d in range(1, N + 1):
          sd = sig[d]
          if sd == 0:
              continue
          for t in range(d, N + 1):
              M[t] = (M[t] + sd * p[t - d]) % mod
      return M


  # -------------------------
  # Scan + summarization
  # -------------------------

  def progression_holds(M: list[int], ell: int, r: int, N_scan: int) -> bool:
      """
      Test M(ell*n + r) == 0 mod ell for all indices <= N_scan.
      Here M is already mod ell.
      """
      t = r
      while t <= N_scan:
          if M[t] % ell != 0:
              return False
          t += ell
      return True


  def run_scan(max_m: int = 49, N_scan: int = 2000, ell_max: int = 97, include_r0: bool = True) -> dict[tuple[int, int], list[int]]:
      """
      Returns dict: (m, ell) -> list of residues r (0<=r<ell) such that
      M_m(ell*n + r) == 0 (mod ell) for all ell*n+r <= N_scan.
      """
      ms = [m for m in range(1, max_m + 1) if m % 2 == 1]
      primes = [p for p in primes_up_to(ell_max) if p >= 5]

      results: dict[tuple[int, int], list[int]] = {}
      for m in ms:
          print(f"\nComputing M_{m}(n) up to n={N_scan} ...")
          for ell in primes:
              M = moments_Mm_up_to_mod(N_scan, m, ell)  # compute mod ell
              good_rs = []
              r_start = 0 if include_r0 else 1
              for r in range(r_start, ell):
                  if progression_holds(M, ell, r, N_scan):
                      good_rs.append(r)
              if good_rs:
                  results[(m, ell)] = good_rs
                  print(f"  Found: m={m}, ell={ell}, r={good_rs}")
      return results


  def summarize_zero_nonzero(results: dict[tuple[int, int], list[int]]) -> None:
      """
      Input: (m, ell) -> [r's]
      Output: splits r=0 and 1<=r<ell as (ell,r)->sorted list of m.
      """
      zero: dict[tuple[int, int], list[int]] = {}
      nonzero: dict[tuple[int, int], list[int]] = {}

      for (m, ell), rs in results.items():
          for r in rs:
              target = zero if r == 0 else nonzero
              target.setdefault((ell, r), []).append(m)

      print("\n=== r = 0 classes ===")
      if not zero:
          print("(none)")
      else:
          for (ell, r), ms in sorted(zero.items()):
              ms.sort()
              print(f"(ell,r)=({ell},{r}): m = {ms}")

      print("\n=== 1 <= r < ell classes ===")
      if not nonzero:
          print("(none)")
      else:
          for (ell, r), ms in sorted(nonzero.items()):
              ms.sort()
              print(f"(ell,r)=({ell},{r}): m = {ms}")


  # -------------------------
  # Certification via Sturm bound
  # -------------------------

  def certify_progression(m: int, ell: int, r: int, prime: int, use_L_power: int = 1, verbose: bool = True) -> dict:
      """
      Certify M_m(ell*n + r) == 0 (mod prime) for all n>=0 by checking up to Sturm bound.

      Level model: Gamma0(4L) where L = ell^use_L_power.
      Sturm bound: B = floor((k/24)*index), k=2m+1.
      We check t = ell*n + r for n=0..B, i.e. indices up to N_max = ell*B + r.

      Computation of M_m(t) is done mod prime for speed.
      """
      L = ell ** use_L_power
      B = sturm_bound_half_integral(m, L)
      N_max = ell * B + r

      if verbose:
          model = f"Gamma0(4*ell^{use_L_power})"
          print(f"Certifying (m={m}, ell={ell}, r={r}) mod {prime} with {model}:")
          print(f"  Sturm bound B={B}, so need indices up to N_max={N_max}")

      M = moments_Mm_up_to_mod(N_max, m, prime)
      for n in range(B + 1):
          t = ell * n + r
          if M[t] % prime != 0:
              if verbose:
                  print(f"  FAIL at n={n}, t={t}, residue={M[t] % prime}")
              return {
                  "status": "FAIL",
                  "m": m, "ell": ell, "r": r, "prime": prime,
                  "L_effective": L, "bound_B": B,
                  "fail_n": n, "fail_t": t, "fail_residue": M[t] % prime
              }

      if verbose:
          print("  PASS")
      return {
          "status": "PASS",
          "m": m, "ell": ell, "r": r, "prime": prime,
          "L_effective": L, "bound_B": B,
          "N_max_checked": N_max
      }


  def run_certifications(tasks: list[dict]) -> None:
      """
      For each task dict {m, ell, r, prime}, run both:
        - L = ell (use_L_power=1)
        - L = ell^2 (use_L_power=2)
      """
      print("\n==============================")
      print("CERTIFICATIONS (Sturm checks)")
      print("==============================\n")
      for t in tasks:
          certify_progression(**t, use_L_power=1, verbose=True)
          certify_progression(**t, use_L_power=2, verbose=True)
          print("-" * 50)


  # -------------------------
  # Main
  # -------------------------

  if __name__ == "__main__":
      # ---------
      # (A) Scan
      # ---------
      MAX_M = 49
      N_SCAN = 2000
      ELL_MAX = 97
      INCLUDE_R0 = True  # set False to reproduce the earlier "nonzero residues only" scan

      results = run_scan(max_m=MAX_M, N_scan=N_SCAN, ell_max=ELL_MAX, include_r0=INCLUDE_R0)

      print("\n\nSummary (grouped by (ell,r) -> list of m):")
      summarize_zero_nonzero(results)

      # ------------------------
      # (B) Sturm certifications
      # ------------------------
      # Put here exactly the progressions you claim as THEOREMS.
      # Keep r=0 ones separate in the paper (as "zero-class congruences") if desired.
      tasks = [
          # nonzero-residue congruences (typical Ramanujan-type classes)
          {"m": 3,  "ell": 7,  "r": 5, "prime": 7},
          {"m": 5,  "ell": 5,  "r": 4, "prime": 5},
          {"m": 7,  "ell": 7,  "r": 5, "prime": 7},
          {"m": 7,  "ell": 11, "r": 6, "prime": 11},

          # optional: r=0 classes (zero-class congruences)
          {"m": 3,  "ell": 11, "r": 0, "prime": 11},
          {"m": 9,  "ell": 5,  "r": 0, "prime": 5},
          {"m": 11, "ell": 11, "r": 0, "prime": 11},
      ]

      run_certifications(tasks)

\end{verbatim}

All computations for this appendix were carried out in Python~3.12 using
the above script.

\paragraph{Overpartition modification.}
The overpartition experiment uses the same code skeleton with only the base
Euler product changed from
$P(q)=(q;q)_\infty^{-1}$ to
$\overline{P}(q)=(-q;q)_\infty/(q;q)_\infty=\eta(2\tau)/\eta(\tau)^2$.
Equivalently, in the Euler product
\[
A(q)=\prod_{r\ge1}(1-q^r)^{-c(r)},
\]
the exponent sequence changes from $c(r)\equiv 1$ (ordinary partitions) to
\[
c(r)=
\begin{cases}
2,& r\ \text{odd},\\
1,& r\ \text{even}.
\end{cases}
\]
In the implementation this affects only the divisor--sum routine: the ordinary
$\sigma_m(n)=\sum_{d\mid n}d^m$ is replaced by the weighted sum
\[
\sigma^{(m)}_{\overline{P}}(n)=\sum_{d\mid n}c(d)\,d^m
=\sum_{d\mid n}d^m+\sum_{\substack{d\mid n\\ d\ \text{odd}}}d^m,
\]
i.e.\ one adds an ``odd--divisor'' contribution.
The moment sequence is then computed by the same convolution formula
\[
M_m^{\overline{\ }}(t)=\sum_{d=1}^{t}\sigma^{(m)}_{\overline{P}}(d)\,
\overline{p}(t-d),
\]
where $\overline{p}(n)$ is computed modulo $\ell$ from the recurrence for the
Euler product $\overline{P}(q)$ (or, equivalently, by forming the product
$\prod_{n\le N}(1+q^n)/(1-q^n)$ modulo $\ell$ up to the required truncation).
All subsequent steps are unchanged: scanning arithmetic progressions,
grouping candidate congruences, and certifying them by checking coefficients up
to the half--integral Sturm bound at a safe level $4\ell^2$.

\end{document}